%% file: AW20draft.tex
\documentclass[10pt,a4paper,reqno]{amsart} 

\usepackage{amsmath}
\usepackage{amssymb}
\usepackage{amsthm} 
\usepackage[initials,msc-links]{amsrefs}
\usepackage{url} 
\usepackage[dvipsnames]{xcolor} 
\usepackage[english=american]{csquotes}
\usepackage{enumerate} 
\usepackage{mathrsfs}
\usepackage{mathtools}  
\usepackage{bbm} 
\usepackage{hyperref}   
\hypersetup{
unicode=false,          
pdftoolbar=true,        
pdfmenubar=true,        
pdffitwindow=false,     
pdfstartview={FitH},    
pdftitle={My title},    
pdfauthor={Author},     
pdfsubject={Subject},   
pdfcreator={Creator},   
pdfproducer={Producer}, 
pdfkeywords={keyword1, key2, key3}, 
pdfnewwindow=true,      
colorlinks=true,       
linkcolor=blue ,          
citecolor=blue ,        
filecolor=magenta,      
urlcolor=Aquamarine           
}

\newtheorem{theorem}{Theorem}
\newtheorem{lemma}[theorem]{Lemma}
\newtheorem{proposition}[theorem]{Proposition}

\theoremstyle{definition}
\newtheorem{definition}[theorem]{Definition}
\newtheorem{remark}[theorem]{Remark}

\newcommand{\ud}{\;\mathrm{d}} 
 
\newenvironment{proofof}[1][Proof]{\noindent \textit{#1.} }{\ \qed}

\newcommand{\Ocal}{\mathcal{O}}

\DeclareMathOperator{\id}{id}

\newcommand{\set}[2]{\left\{\, #1 \  \textup{\textbf{:}}\  #2 \,\right\}}

\newcommand{\dpr}[1]{\langle #1 \rangle}

\newcommand{\dd}{\;\mathrm{d}}

\newcommand{\R}{\mathbb{R}}

\newcommand{\loc}{\mathrm{loc}}

\newcommand{\sbullet}{\begin{picture}(1,1)(-0.5,-2)\circle*{2}\end{picture}}
\newcommand{\frarg}{\,\sbullet\,}

\newcommand{\cf}{{\mathbbm 1}}

\newcommand{\vphi}{\varphi}

\newcommand{\Reals}{\mathbb R}

\makeatletter
\renewcommand*\env@matrix[1][*\c@MaxMatrixCols c]{%
  \hskip -\arraycolsep
  \let\@ifnextchar\new@ifnextchar
  \array{#1}}
\makeatother

\DeclareMathOperator{\Id}{id}

\mathtoolsset{showonlyrefs}

\title[Debye screening]{Debye screening {for} the stationary Vlasov-Poisson equation in interaction with a point charge}
 
\author[A.~Arroyo-Rabasa]{Adolfo Arroyo-Rabasa}
\address{A.A.-R.: Mathematics Institute, University of Warwick, Coventry CV4 7AL, UK.}
\email{\href{mailto:adolfo.arroyo-rabasa@warwick.ac.uk}{adolfo.arroyo-rabasa@warwick.ac.uk}}
\author[R.~Winter]{Raphael Winter}
\address{R.W.: Universit\'{e} de Lyon,
	43 Boulevard du 11 Novembre 1918, 69100 Villeurbanne, France}
\email{\href{mailto:raphael.winter@ens-lyon.fr}{raphael.winter@ens-lyon.fr}}

\date{\today}


\begin{document}

	\begin{abstract} 
		We prove that the Debye screening length emerges in an infinitely extended plasma {with uniform background} described by the nonlinear Vlasov-Poisson equation, interacting with a point charge. While screened stationary states as well as the time-dependent problem are well-studied for the linearized equation, there are few rigorous results on screening in the nonlinear setting. As such, the results presented here  {cover} the stationary case under the assumptions predicted by the linearized theory. 
	\end{abstract}
\maketitle



\section{Introduction}

In this paper, we show  screening for the interaction of a point charge with a plasma described by the nonlinear Vlasov-Poisson equation. The so-called Debye screening is essential to both the physical and mathematical theory of plasmas: 
the principle states that, in spite of the scale invariance of the Coulomb interaction potential
\begin{align}\label{eq:Coulomb}
	\phi_c(x) = \frac{\kappa_0}{|x|}, \quad \kappa_0>0,
\end{align}
 the resulting  effective potential of a single particle has a well-defined length scale.
 More precisely, the effective potential is of Yukawa type:
 \begin{align} \label{eq:Yukawa}
 	\phi_{\operatorname{eff}}(x)= \frac{\kappa_1}{\left(|x|/L_D\right)} e^{- \frac{|x|}{L_D}},
 \end{align}
where $L_D$ is the Debye length of the plasma. 
Hence, in contrast to \eqref{eq:Coulomb}, the effective potential emerging from the dynamics is of short range. The screening principle has become a ubiquitous concept in plasma physics, and in the mathematical theory of plasmas it is  often incorporated directly into the model (e.g. \cites{bedrossian_landau_2018,han-kwan_quasineutral_2011}). For the grand-canonical distribution of  particles in space, Debye screening of correlations has been proved rigorously in \cite{brydges_debye_1980}.

Our goal here is to provide a rigorous proof for the Debye screening principle starting from a nonlinear Vlasov-Poisson equation. While the techniques employed here are relatively simple, it appears to be the first result of this type to the knowledge of the authors. 

 We consider an infinitely extended Vlasov-Poisson equation {with a uniform} background distribution of ions in interaction with a point charge at the origin 	\begin{equation}\label{eq:Vlasov}
 	\begin{aligned}
 	\partial_tf+ v  \cdot \nabla_x f - \theta_V \nabla_x Q \cdot \nabla_v f &=   0, \quad f(0,x,v)=f_0(v)\\
 		\rho[f](t,x) &= \int_{\Reals^3} f(t,x,v) \dd v \\
 	-\Delta_x Q(t,x) &=  (N\theta_V)(\rho[f]-1)  + \theta_P \delta_0.
 	\end{aligned}
 	\end{equation} 
 	Here{,}
 	\begin{itemize} 
 		\item $N$ is the number density of electrons per unit volume,
 		\item $\theta_V$ denotes the charge of those electrons, and
 		\item $\theta_P$ is the size of the point charge at the origin.
 	\end{itemize}	
Moreover, a {uniform} background distribution of ions is assumed, so the self-consistent potential is generated by $\varrho[f]-1$. We will restrict ourselves to repulsive plasma-point charge interaction in this paper, i.e., $\theta_P \theta_V>0$.
In physics, the system \eqref{eq:Vlasov} is used both for attractive and repulsive plasma-point charge interaction ---mainly plasma-electron and plasma-ion interaction (cf. Section 1.7 in \cite{goldston_introduction_1995}). Debye screening of plasma-electron interaction is covered by our result, and is also the case considered in \cite{brydges_debye_1980}. Physical applications of \eqref{eq:Vlasov} for repulsive interactions are for example the effective interaction of electrons in a plasma (cf. Section 8.2 in \cite{schram_kinetic_1991}), the study of fluctuations in plasmas (cf. \cite{rostoker_fluctuations_1961,rostoker_superposition_1964}), and the kinetic theory of plasmas (see \cite{piasecki_stochastic_1987}).
The model above is also widely used in the case of plasma-ion interaction. In this case, the interaction is typically attractive, and thus not covered by our result. An exemplary application is the analysis of the friction coefficient of an ion passing through a plasma (cf. \cite{boinefrankenheim_nonlinear_1996,peter_energy_1991}). We believe that, due to its physical relevance in plasma-ion interaction, this case deserves a thorough investigation in future works.

	We proceed with our analysis of \eqref{eq:Vlasov}.
 	Choosing an appropriate length scale, we may set~{$N\theta_V^2=1$}. Then{,} for a plasma with a well-defined temperature $T>0$, the Debye length  $\lambda_D$ is given by
 	\begin{align} \label{eq:Debye}
 	\lambda_D= \sqrt{\frac{T}{N\theta_V^2}}= \sqrt{T}.
 	\end{align}
 	We choose units such that the mass of an electron and the Boltzmann constant are one. We shall as well assume the normalization
 	\begin{align}
 	\int_{\Reals^3} f_0(v) \ud{v} =1.
 	\end{align}
 	If the distribution $f_0$ is at a thermodynamic equilibrium with temperature $T$, then $f_0$ is the Maxwellian $M_T(v)$ given by
 	\begin{align} \label{eq:Maxwellian}
 		M_T(v) = \frac{1}{(2\pi T)^\frac32} e^{-\frac{|v|^2}{2T}}.
 	\end{align}
 	As observed by linearization (cf. \cite{goldston_introduction_1995,lifshitz_course_1981}), Debye screening can be predicted by looking at the stationary states of \eqref{eq:Vlasov}.
	Let the distribution $f_0$ be continuous and radial, i.e.,
	\begin{align} \label{eq:radial}
		f_0(v)= F \big(\frac12 |v|^2\big),\quad  f_0 \in C^1(\Reals^3),
	\end{align}
and assume that $f_0$ is algebraically decaying. More precisely, we assume:
	\begin{align}\label{algebraic}
		|F(r)| + |F'(r)|\leq \frac{C}{1+r^2} \, ,\quad F\in C^1(0,\infty) .
	\end{align}
	Then, the solution to the linearized system (cf. Section 1.7 of \cite{goldston_introduction_1995}) satisfies the screening estimates:
	\begin{gather}
		0\leq Q(x)		\leq \frac{C  \theta  e^{-\frac{|x|}{\lambda_D}}}{|x|}\label{est:Q},\\
	|1-\rho[f]|	\leq C \theta e^{-\frac{|x|}{\lambda_D}},\label{est:rho}
	\end{gather}
	where $\theta=\theta_P {\cdot }\theta_V$.
	We remark that the  condition of radial symmetry \eqref{eq:radial} is necessary since  the linearized theory predicts algebraically decaying  stationary states for non-radial functions $f_0$.
	Furthermore, linearized stability of the {uniform} stationary state has been shown to hold under the Penrose condition (see \cites{glassey_time_1994,glassey_time_1995,lancellotti_glassey-schaeffer_2015}). Quantitative estimates for the linearized system with a point charge can be found in \cite{vogt_debye_2015}. In the radially symmetric case, the Penrose condition (see  \cite{penrose_electrostatic_1960}) can be written as 
	\begin{align} \label{eq:monotone}
		F'(r) < 0 \quad \forall r > 0.
	\end{align}
	Under the conditions \eqref{eq:radial},~\eqref{algebraic} and~\eqref{eq:monotone}, we show the existence of screened radial stationary states of the nonlinear Vlasov-Poisson equation \eqref{eq:Vlasov}.  In the time-dependent problem, we consider homogeneous stationary initial data $f(0,x,v)= f_0(v)$. For the stationary case, $f_0$ becomes the boundary condition at $|x|\rightarrow \infty$, i.e., we impose that the plasma is unperturbed at infinity.
	
	 The precise statement of our main theorem is the following:
	\begin{theorem} \label{Theorem}
		Let $f_0$ be a radial probability density as in~\eqref{eq:radial}. Assume that $F$ satisfies conditions \eqref{algebraic} and \eqref{eq:monotone}.
		Then, for every positive $\theta>0$ there exists a radial weak solution $f$ (cf. Definition~\ref{def:weak})  to
		\begin{equation} \label{eq:main}
			\begin{aligned}
			v \cdot \nabla_x f - \nabla_x  Q \cdot \nabla_v f &=   0\\
			-\Delta_x Q &=  (\rho[f]-1)  + \theta \delta_0 \, ,
			\end{aligned} 
		\end{equation}
		where 				\[
		\rho[f](x) =\rho (x)= \int_{\Reals^3} f(x,v) \dd v,
		\]
		and $f$ is subject to the boundary condition
		\[
			\lim_{|x|\rightarrow \infty} f(x,v)=f_0(v).
		\]

		
		Moreover, the effective potential $Q$ of the solution and the spatial density $\rho[f]$ satisfy the exponential bounds \eqref{est:Q} and \eqref{est:rho}. 
		
		Furthermore, there is an explicit formula for the characteristic length scale of decay (see \eqref{eq:sigma}, \eqref{est:QLemma}). Whenever there is a well-defined Debye-length ($f_0$ Maxwellian with temperature $T$), the length scale is indeed given by $\lambda_D$ as defined in \eqref{eq:Debye}.
	\end{theorem}
	\begin{remark}[Uniqueness]
		We remark that we do not assume smallness of the point charge, i.e.{,} $0 < \theta \ll 1$. On the other hand, the theorem above does not prove uniqueness of the screened stationary state. 
		Multiple solutions (possibly non-radial) might exist, particularly for large perturbations $\theta \gg 1$. 
	\end{remark}
	The time-dependent Vlasov-Poisson equation of an infinitely  extended plasma in interaction with a point charge has been studied in a number of papers \cites{caprino_time_2015,chen_asymptotic_2015,chen_global_2015,crippa_lagrangian_2018,desvillettes_polynomial_2015,li_3-d_2017,marchioro_cauchy_2011,caprino_plasma-charge_2010}. 
	There are two features of the system which make it particularly hard to obtain mathematically rigorous results. Firstly, the system has infinite mass and infinite energy, so they cannot be used as conserved quantities.  Secondly, the macroscopic charge contained in the point particle never disperses, so there is no damping in time for the self-consistent field around it. Therefore, the classical theory of the Vlasov-Poisson equation (see, e.g., \cite{bardos_global_1985}) does not apply.

	In the works \cites{caprino_time_2015,chen_asymptotic_2015,chen_global_2015,crippa_lagrangian_2018,desvillettes_polynomial_2015,li_3-d_2017,marchioro_cauchy_2011,caprino_plasma-charge_2010} mentioned above, global well-posedness of the system {is established} under the assumption that velocities are  compactly supported and that the plasma density goes to zero close to the point charge. This results in finite speed of propagation within the system. Both conditions can be shown to be preserved by the system over time, and the authors obtain a polynomial bound on the growth of moments.
	It is likely that Debye screening is a key ingredient to generalizing the results on the time-dependent problem to more general initial data and obtaining information on the asymptotics for $t\rightarrow \infty$. Developing new techniques to address the intricate interplay of the self-consistent potential and the induced characteristics is  a challenging objective for future research.

	
		The stationary case is more feasible in the sense that it bypasses/circumvents some of the rigorous mathematical impediments discussed above. The key observation is that, once the system has reached a stationary state, the characteristics are fixed. From this point on, our approach is  reminiscent of the technique used for solitons in \cite{strauss_existence_1977}.

	\input{General}
	\input{Decay}

\bibliography{Arroyo-Winter}
\bibliographystyle{amsplain}

\end{document}

%% file: General.tex
\section{Proof of Theorem~\ref{Theorem}}
\begin{definition}[Weak solution] \label{def:weak}
	Let $f\in C(\Reals^3\times \Reals^3) \cap L^1_\loc(\R^3;L^1(\R^3))$ be a continuous function which satisfies  $f(x,\frarg) \in L^1(\R^3)$ for all points $x \in \R^3$, and  
	\[
		\int_{K} \int_{\Reals^3}|f(x,v)| \ud{x} \ud{v} < \infty
	\]
	for every compact set $K\subset \Reals^3$.
	
	Let $\theta > 0$ be a positive real number and let $\rho$ be the spatial density of $f$ given by 
	\begin{align}
		\rho(x)= \int_{\Reals^3}f(x,v) \ud{v}.
	\end{align}
		We say $f$ is a weak solution of~\eqref{eq:main} if and only if
	\begin{enumerate}[1)]
		\item $f$ agrees with $f_0$ at infinity:
	\begin{align} \label{def:boundary}
		\lim_{|x|\rightarrow \infty } f(x,v)= f_0(v) \quad \text{for all $v\in \Reals^3$};
	\end{align}
	\item there exists $Q\in W^{1,1}_{\loc}(\Reals^3)$ such that
	\begin{align*}  
	-\int_{\Reals^3 }\Delta_x \psi(x) Q(x) \ud{x}=  \int_{\Reals^3} (\rho-1)\psi \dd x  + \theta \psi(0)
	\end{align*}
	for all $\psi \in C_c^\infty(\R^3)$;
	\item and, both $f$ and $Q$ verify the integral equation  
	\begin{align*} 
	\int_{\R^3} \int_{\R^3} f (v \cdot \nabla_x \vphi) \dd v \dd x =  \int_{\R^3} \int_{\R^3}  f (\nabla_x Q \cdot \nabla_v \vphi) \dd v \dd x
		\end{align*}
	among all test functions $\vphi \in C^\infty_c(\R^3 \times \R^3)$. 
	
		\end{enumerate}
\end{definition}

\subsection*{Formal derivation} We commence by deriving a \emph{sufficient} condition to provide a solution to the system~\eqref{eq:Vlasov}. Firstly, we rest on the formal  principle that the Hamiltonian 
\begin{align}\label{eq:H}
	H(x,v):= Q(x)+ \frac12 |v|^2
\end{align}
is constant along the characteristics of the system \eqref{eq:Vlasov}. This motivates us to look for a solution of the form:
\begin{align} \label{eq:structural}
	f(x,v) = F(H(x,v)).
\end{align}
{Notice that the \emph{necessity} for this structural condition is only present provided that the level sets are precisely the orbits of the characteristic flow}.
Now, due to the boundary condition $f \equiv f_0$ at infinity, we conclude that $F$ and $f_0$ are related through the following expression:
\begin{align}
	\lim_{|x|\rightarrow \infty} f(x,v) &= \lim_{|x|\rightarrow \infty}  F(H(x,v))\\
	&=F(\frac12 |v|^2) = f_0(v)  \quad \text{for all $v \in \R^3$.} \label{eq:Ff0}
\end{align}
\begin{remark}
	In the case of attractive interaction between point charge and plasma, we cannot construct a solution using only the structural assumption \eqref{eq:structural} and the boundary condition. In this case, $Q$ could be negative around the point-charge. This gives rise to the difficulty that $F(r)$ is only determined by \eqref{eq:Ff0} for non-negative values $r\geq 0$, but might have a negative argument in \eqref{eq:structural}. This difficulty stems from orbits which are confined to a neighborhood of the point charge, and are therefore disconnected from the imposed boundary condition \eqref{def:boundary}.
\end{remark}
Consider the function 
\begin{align}\label{def:g}
	g(y):= \int_{\Reals^3} F(y+\frac12 |v|^2) \ud{v}, \quad y \ge 0.
\end{align}
Then, if $f$ is as in~\eqref{eq:structural}, the spatial density 
can be expressed as
\begin{align}	
	\rho[f](x) &= \int_{\Reals^3} F( Q(x) + \frac12 |v|^2) \ud{v} =  g(Q(x)).
\end{align}
Re-inserting the formula for $\rho$ above back into the definition of $Q$ (c.f. Definition~\ref{def:weak}), we obtain that $Q$ must satisfy the  equation
\begin{align} \label{eq:fixed}
-\int_{\R^3} Q \Delta \psi \dd x &= \int_{\R^3} (g(Q(x))-1) \psi \dd x + \theta \psi(0),
\end{align}
for all $\psi \in C_c^\infty(\R^3)$. {We shall henceforth focus in showing the existence of radial solutions to equation \eqref{eq:fixed}.} Following the convexity ideas in~\cites{strauss_existence_1977}, we split the term in the right-hand side above by extracting a mass $\sigma Q$ and adding the  operator defined by 
\begin{align} \label{def:B}
B_\sigma[Q] := g(Q_+)-1 + \sigma Q_+,
\end{align}
where $Q_+ = Q\cdot \cf_{Q\geq 0}$.
(We will chose the constant $\sigma>0$ below.)
With this definition, equation \eqref{eq:fixed} can be re-written as the following fixed-point problem
\begin{align}\label{eq:fixedfinal}
Q = (\sigma \Id- \Delta)^{-1}(B_\sigma[Q]+\theta \delta_0).
\end{align}



%

To show well-posedness of this equation, we make the following observation on the boundedness of the (non-linear) operator $B_\sigma$.
\begin{lemma}\label{lem:lem}
	Let $F$ satisfy the assumptions \eqref{algebraic} and \eqref{eq:monotone}.
	Recall the function $g$ defined in \eqref{def:g}.
	Then, $g$ is  strictly convex and satisfies 
	\begin{equation}\label{eq:sandwich}
		0 \le g(y) \le 1 = g(0) \quad  \qquad \text{for all $y \ge 0$.}
	\end{equation} 
	For $\sigma > 0$ chosen as 
	\begin{align} \label{eq:sigma}
		-\sigma:= g'(0) \coloneqq \lim_{r \to 0^+} g'(r),
	\end{align}
	the mapping $B_\sigma$ (cf. \eqref{def:B}) satisfies
	\begin{align}
		|B_\sigma[Q]| &\leq C \frac{|Q|^2}{1  + |Q|}  \quad \text{for all $Q\in L^1_{\mathrm{loc}}$}. \label{est:BQ}
	\end{align}
	 Moreover, $B_{\sigma}$ is continuous as an operator
	\begin{align}
		B_\sigma  : L^q \rightarrow L^{p}
	\end{align} 
	 for all $1\leq p\leq q\leq 2p$. Furthermore, we have the estimates
	\begin{align}
			\|B_\sigma [Q]\|_{L^p} &\leq C \|Q\|^\frac{q}p_{L^q},\label{eq:Bproperty}
	\end{align}
	and the range of $B_\sigma $ consists of non-negative functions:
	\begin{align}\label{eq:Bnonneg}
		B_\sigma [Q]\geq 0.
	\end{align}
\end{lemma}
\begin{remark}[Weak Penrose condition]
	The condition \eqref{eq:monotone} can be relaxed to directly imposing the inequalities \eqref{eq:gP} and \eqref{eq:gPP} obtained below. 
\end{remark}
\begin{proof}
	Changing variables we re-write $g$ as
	\begin{align}
		g(y) = 4\pi \int_0^\infty  \sqrt{2r} F(y+r) \ud{r}.
	\end{align}
	Using \eqref{algebraic} and \eqref{eq:monotone} this implies:
	\begin{align}
		g'(y) &= -4\pi \int_0^\infty  {\frac 1{\sqrt{2r}}} F(y + r) \ud{r}<0 \label{eq:gP}\\
		g''(y) &= -4\pi \int_0^\infty  \frac{1}{\sqrt{2r}} F'(y+r) \ud{r}>0.\label{eq:gPP}
	\end{align}
 By construction $g$ is non-negative and $g(0) = 1$,  hence~\eqref{eq:sandwich} follows from~\eqref{eq:gP}. Inequality~\eqref{eq:gPP} implies that $g$ is strictly convex with $\|g''\|_\infty < \infty$.
We define 
$$-\sigma \coloneqq g'(0) \coloneqq \lim_{r \to 0^+} g'(r).$$  Using~\eqref{eq:gP}-\eqref{eq:gPP}, we obtain the bound $\|g'\|_\infty \le \sigma < \infty$, so that $g$ is  globally Lipschitz. Another relevant consequence of the convexity of $g$ is that its first-order approximation at zero $1 - \sigma r$ is also a (one-sided) sub-differential of $g$, that is,
\[
b(r) \coloneqq g(r) - 1 + \sigma r \ge 0 \qquad (r \ge 0).
\]
In particular, by the definition of $B_\sigma$ (see~\eqref{def:B}), we verify that $B_\sigma[Q]\ge 0$ for all $Q \in L^q$.

Furthermore, on $[0,\infty)$ the function $b$ is globally Lipschitz  since $g$ is globally Lipschitz. Notice that for small values $t,s > 0$, a better estimate of the variation is given by
\begin{align*}
	|b(t) - b(s)| & = |\int_s^t b'(\omega) \dd \omega| \\
	& \leq  \int_s^t \int_0^\omega |b''(\rho)| \dd \rho \ud{\omega} \le \|g''\|_\infty \max\{t,s\}|t-s|.
\end{align*}
This calculation and the fact that $b$ is Lipschitz convey the bound 
\[
|b(t) - b(s)| \le \Big((\|g''\|_\infty  \max\{t,s\})  \wedge \mathrm{Lip}(b) \Big) |t -s|.
\]
For $s,t \in [0,1 + \mathrm{Lip}(b)\|g''\|_\infty^{-1}]$ we verify that 
\[
\max\{t,s\} \le (3 + 2 \mathrm{Lip}(b)\|g''\|_\infty^{-1})\frac{s + t}{1 + s + t}.
\]
On the other hand, if $s \ge 1$ or $t \ge 1$, then 
\[
\frac{1}{2} \le  \frac{s + t}{1 + s + t}.
\]
Thus, setting $C \coloneqq \max\{2\mathrm{Lip}(b),\|g''\|_\infty (3 + 2\mathrm{Lip}(b)\|g''\|_\infty^{-1})\}$ we conclude that
\[
|b(t) - b(s)|\le C \bigg( \frac{s + t}{1 + s + t}\bigg) |t - s| \qquad (t,s \ge 0).
\]

In particular, by the definition of $B_\sigma$, we may make use of the estimate
\begin{align}
	|B_\sigma[Q]-B_\sigma[P]| &\leq C \bigg(\frac{P_+ + Q_+}{1 + P_+ + Q_+}\bigg) |P_+-Q_+| \quad  \\
							&\leq C \bigg(\frac{|P|+|Q|}{1 + |P|+|Q|}\bigg) |P-Q| \quad \text{for all $P,Q \in L^1_{\mathrm{loc}}$}\label{est:Bprime}.
\end{align}
Setting $P=0$ yields \eqref{est:BQ}.
This gives $B_\sigma[Q] \in L^p$ for all $Q \in L^q$ and  $1 \le p \le q \le 2p$. Now, let us fix $P,Q \in L^q$ and set 
\[
U \coloneqq \frac{|P| + |Q|}{1 + |P| + |Q|},
\]
which belongs to $L^r$ for all $q \le r \le \infty$.

Observe that $s \coloneqq q/p \ge 1$ has the  dual exponent $s' = q/(q-p)$. The assumption $p\leq q\leq 2p$ implies that $q \leq ps'$. Hence, using that $|U|\leq 1$, this conveys the bound
\begin{align}
	|U|^{ps'} \leq 	|U|^q.
\end{align}
Thus, H{\"o}lder's inequality yields the estimate 
\begin{align*}
	\int |B_\sigma[P]-B_\sigma[Q]|^p & \le C \int |P-Q|^p U^p \\
	& \le C \bigg(\int |P-Q|^q \bigg)^{\frac pq} \bigg(\int U^q\bigg)^\frac{q-p}{q} \\
	& \le C \|P-Q\|_{L^q}^p \big(\|P\|_{L^q}^{q-p} + \|Q\|_{L^q}^{q-p}\big).
\end{align*}
This shows that $B_\sigma  : L^q \to L^p$ is a continuous operator for all $1 \le p \le q \le 2p$. Moreover, setting $P \equiv 0$ in the estimate above we conclude that 
\[
\|B_\sigma[Q]\|_{L^p} \le C \|Q\|_{L^q}^{\frac{q}{p}}.
\]
This finishes the proof.
\end{proof}
Further we use that the fundamental solution to the operator $(\sigma \Id- \Delta)^{-1}$ decays exponentially at infinity.
\begin{proposition} \label{prop:fundamental}
	Let $\sigma > 0$ and let $\Phi_\sigma$ be the fundamental solution defined by 
	\begin{align}
		\Phi_\sigma  \coloneqq (\sigma \Id- \Delta)^{-1} \delta_0.
	\end{align}
	Then $\Phi_\sigma \in C^\infty(\R^3 \setminus \{0\})$ is explicitly given by
	\begin{align}\label{est:zeta}
		\Phi_\sigma (x) = \frac{e^{-\sqrt{\sigma}|x|}}{4\pi|x|},
	\end{align}
and $\Phi_\sigma  \in  W^{1,q}(\R^3)\cap L^p(\Reals^3)$ for all $q \in [1,\frac32)$ and $p\in [1,3)$.

	Furthermore, the operator $(\sigma \Id- \Delta)^{-1}$ satisfies the estimate
	\begin{align} \label{est:inverse}
		\|(\sigma \Id- \Delta)^{-1} u\|_{H^1} \leq C_\sigma  \|u\|_{L^2}.
	\end{align}
\end{proposition}
\begin{proof}
	The identity \eqref{est:zeta} can be verified explicitly. 
	The estimate \eqref{est:inverse} follows using the fact that the quadratic form $w \mapsto \dpr{(\sigma \id - \Delta)w,w}_{L^2}$ induces an equivalent norm to the one  of $H^1$.
\end{proof}
The next ingredient is the following key compactness criterion for radial Sobolev maps. Here, we make vital use of rotational symmetry, which prevents a lack of compactness due to translations. The following properties of radial functions can be found in \cite{lions_symetrie_1982}.
\begin{proposition}\label{prop:cpt}
	For $2<q<q^*$, the embedding
	\begin{align}
		H^1_{\operatorname{rad}}(\Reals^3) \rightarrow L^q(\Reals^3)
	\end{align}
	is compact. 
	
	Furthermore, there exists a constant $c>0$ such that every function $u\in H^1_{\mathrm{rad}}(\Reals^3)$ satisfies the decay estimate
	\begin{align} \label{lions:pointwise}
		|u(x)|\leq c \|u\|_{H^1} \cdot |x|^{-1}.
	\end{align}
\end{proposition}
Combining the statements above we obtain the existence of a fixed point of equation~\eqref{eq:fixed} by means of Schaefer's Fixed Point Theorem:

\begin{theorem}[Schaefer's Theorem]\label{thm:S}
	Let $X$ be a Banach space and let $T : X \to X$ be a continuous and compact mapping. If the set 
	\begin{align} \label{def:O}
	\Ocal \coloneqq \set{x \in X}{x = \lambda T(x) \text{ for some $\lambda \in [0,1]$}}, 
	\end{align}
	is bounded, then $T$ has a fixed point in $X$. 
\end{theorem}
We now prove the existence of a unique solution $Q$ to \eqref{eq:fixedfinal}.
\begin{lemma} \label{lem:fixed}
	Let $2<q<3$ and let $\sigma$ be the constant from Lemma~\ref{lem:lem}. Then, there exists a radially symmetric  solution ${Q}\in L^q_{\mathrm{rad}}(\Reals^3)$ to the fixed point equation
	\begin{align}\label{eq:fixedLemma}
	Q = (\sigma \Id- \Delta)^{-1}(B_\sigma [Q]+ \theta \delta_0).
	\end{align}
	The solution $Q$ is unique and independent of the choice of $q\in (2,3)$. Further
	$Q$ satisfies $Q\in C(\Reals^3\setminus \{0\})\cap L^2_{\mathrm{rad}}$ and
	\begin{align}\label{Qpositive}
		Q(x)\geq \theta \Phi_{\sigma} (x) > 0, \quad \text{for $x\in \Reals^3\setminus \{0\}$}.
	\end{align}
\end{lemma}
\begin{proof}
	First we observe that the operator on the right-hand side is compact and continuous {when defined as a mapping} $A$
	\begin{align}
	A : L^q_{\mathrm {rad}} &\rightarrow L^q_{\mathrm {rad}} \\
	Q  &\mapsto  (\sigma \Id- \Delta)^{-1}(B_\sigma [Q]+ \theta \delta_0),
	\end{align}
	This follows since $B_\sigma $ and $(\sigma \Id- \Delta)^{-1}$ are continuous as maps from $L^q$ to $L^2$ and $L^2$ to $H^1$ respectively. {Here, we also use that } $\Phi_\sigma  = (\sigma \Id- \Delta)^{-1} (\delta_0)\in L^q$ due to the assumption $2<q<3$ and Proposition~\ref{prop:fundamental}. Compactness is a direct consequence of   Proposition~\ref{prop:cpt}, whilst it is clear that radial symmetry is preserved by all operators involved.
	

	Let us set  $X = L^q_{\mathrm{rad}}$ and define $T[Q] \coloneqq (\sigma \Id- \Delta)^{-1}(B_\sigma[Q]+\theta \delta_0) $. {In order to show the existence of fixed point solution it suffices to show (cf. Theorem~\ref{thm:S}) that the set of solutions to the equations}
	\begin{align} \label{eq:lambda}
	Q_\lambda  &= \lambda (\sigma \Id- \Delta)^{-1}(B_\sigma[Q_\lambda ]+\theta \delta_0)=\lambda T[Q_\lambda ]
	\end{align}
	is $L^q$-bounded uniformly in $\lambda \in [0,1]$.   
	To this end, assume that $Q_\lambda $ is such a solution and introduce 
	\begin{align}
	R_\lambda  \coloneqq Q_\lambda - \lambda  \theta \Phi_\sigma.
	\end{align}
	Arguing as above, $R_\lambda \in L^2_{\mathrm{rad}}$, and $R_\lambda$ solves
	\begin{align} \label{eq:weakly}
	R_\lambda   &= (\sigma \Id- \Delta)^{-1}\lambda B_\sigma [R_\lambda  +  \lambda \theta \Phi_\sigma ].
	\end{align}
	Hence $R_\lambda \geq 0$ is non-negative. Using the definition of $B_\sigma$ (cf. \eqref{def:B}), the equation above can be rewritten as $R_\lambda \in L^2$ being a weak solution to
	\begin{align}
	- \Delta R_\lambda  = \lambda (g(R_\lambda +\lambda \theta \Phi_\sigma) -1) -(1-\lambda )\sigma R_\lambda  + \lambda^2 \sigma \theta  \Phi_\sigma. 
	\end{align}
	We then use the a comparison principle to estimate $R_\lambda$. More precisely, let
	$\phi(x)= \frac{1}{4\pi |x|}$ be the fundamental solution of the Laplace equation. Then we can estimate
	\begin{equation}\label{est:R1}
	\begin{aligned}
	0\leq R_\lambda  &=  \phi * \left(\lambda(g(R_\lambda +\lambda \theta \Phi_\sigma) -1) -(1-\lambda )\sigma R_\lambda  + \lambda^2 \sigma \theta \Phi_\sigma\right) \\
	&\leq  \lambda^2 \sigma \theta \phi *   \Phi_\sigma \leq   \sigma\theta \phi *   \Phi_\sigma, 
	\end{aligned}
	\end{equation}
	since $g(y)\leq 1$ (cf. \eqref{eq:sandwich}), and $\lambda \in [0,1]$ by assumption.
	
	We recall the boundedness for Riesz potentials due to Sobolev (cf. \cite{sobolev_theorem_1938})
	\[
		\|I_\alpha f\|_{L^{\frac{np}{n-\alpha p}}} \le C_p\|f\|_{L^p}, \qquad f \in L^p(\R^n), \quad 1 < p < \frac n\alpha.
	\] 
	Here $\alpha$ is a real number in $(0,n)$ and the $\alpha$-Riesz potential of $f$ is defined as
	\[
		I_\alpha f(x) = c_\alpha \int_{\R^n} \frac{f(z)}{|x - z|^{n-\alpha}} \dd z,
	\]
	where $c_\alpha$ is a normalizing constant. 
	
	In our context, we have $n = 3$ and by the definition of the fundamental solution $\phi$, there exists $c>0$ such that $\phi \ast \Phi_\sigma = c I_2 \Phi_\sigma$.  Therefore, the boundedness of the $2$-Riesz potential with $p = \frac{12}{11}$ yields the estimate 
	\begin{align} \label{est:Riesz}
	\|\phi * \Phi_\sigma \|_{L^4} \leq C\|\Phi_\sigma\|_{L^\frac{12}{11}},
	\end{align}
where $C > 0$ is a constant independent of $\lambda\in [0,1]$.
Then, the bounds~\eqref{est:R1}-\eqref{est:Riesz} imply the a priori bound:
	\begin{align} \label{est:L4}
	\|R_\lambda \|_{L^{4}} \leq C \|\Phi_\sigma\|_{L^{\frac{12}{11}}}.
 	\end{align}
	 On the other hand, we may then re-write~\eqref{eq:weakly} as
	\begin{align} \label{eq:Rfinal}
	R_\lambda   &= \lambda  \Phi_\sigma * B_\sigma [R_\lambda  + \lambda  \theta \Phi_\sigma ].
	\end{align}
	Using \eqref{est:BQ} and the fact that $\Phi_\sigma\in L^1$, we find that
	\begin{equation} \label{est:L2}
		\begin{split}
	\|R_\lambda \|_{L^2} & \leq C\|B_\sigma[R_\lambda + \lambda \theta \Phi_\sigma]\|_{L^2} \\
	& \le C(\|R_\lambda\|_{L^{4}}^2 + \|\Phi_\sigma\|_{L^2}),
	\end{split}
	\end{equation}
for some $C> 0$ independent of $\lambda\in [0,1]$. Recalling that $\Phi_\sigma \in L^p$ for all $1 \le p < 3$, interpolation of the bounds~\eqref{est:L4}-\eqref{est:L2} yield the uniform $L^q$-bound
	\[
		\|R_\lambda\|_{L^q} \le C, \qquad q\in [2,4],
	\]
	for some $C> 0$ which is independent of $\lambda \in [0,1]$.
	This shows that the set $\Ocal$ (cf. \eqref{def:O}) is $L^q$-uniformly bounded for each $2 < q < 3$. Hence, Schaefer's Theorem yields the existence of a weak solution $Q\in L^q_\mathrm{rad}$ of
	\eqref{eq:fixedLemma}. 
	
	For uniqueness, we recall that every solution of~\eqref{eq:fixedLemma} satisfies $Q\in L^2_{\mathrm{rad}}$. Therefore, it is sufficient to prove two solutions
	$Q,P\in L^2_{\mathrm{rad}}$  coincide. To this end, we observe the difference $S=Q-P$ satisfies
	\begin{align}
		-\Delta S = g(Q) - g(P),
	\end{align}
	in the weak sense. Since $g\in L^\infty$, we know that $S\in C(\Reals^3)$. Now, consider the sets
	\begin{align}
		U^+= \{x \in \R^3: S(x)>0\}, \quad U^-= \{x \in \R^3: S(x)<0\}.
	\end{align} 
	We claim that neither $U^+$ nor $U^-$ can contain a bounded, non-empty connected component. We demonstrate the argument for $U^+$. Assume that $\emptyset\neq U^+_c$ is a bounded connected component of $U^+$. From this and the monotonicity of $g$, we infer the differential inequality
	\begin{align} 
		- \Delta S = g(Q)- g(P) \leq 0,\quad  S> 0, \quad \text{on $U^+_c$},
	\end{align}
	By the continuity of $S$ it holds $\partial U^+_c \subset \{x\in \Reals^3: S(x)=0\}$, so the maximum principle implies that $S\equiv 0$ on $U^+_c$. This contradicts the positivity of $S$ on $U^+_c$. We can therefore conclude that $U^+$ and $U^-$ do not contain connected bounded components. Since $S$ is radial 
	 this  implies $S\geq 0$ on $\Reals^3$ or $S\leq 0$ on $\Reals^3$. Without loss of generality let the former be the case. Then we have $S \in L^2_\mathrm{rad}$ a continuous function satisfying
	\begin{align}
		- \Delta S \leq 0, \quad S\geq 0, \quad \text{on $\Reals^3$}.
	\end{align}
	By the maximum principle $\max_{x: |x|\leq R} S(x) = \max_{x: |x|= R} S(x)$ for every $R>0$. On the other hand $S\in (C\cap L^2_{\mathrm{rad}})(\Reals^3)$, so $\liminf_{R\rightarrow \infty} \max_{x: |x|= R} S(x) = 0$.
	This shows that $S\equiv 0$ and hence $Q\equiv P$. This proves the sought uniqueness.
	
	To show strict positivity of the solution $Q \in L^q_{\mathrm{rad}}$ we represent $Q$ as
	\begin{align}
		Q = \Phi_\sigma * \left( B_{\sigma}[Q]+ \theta \delta_0 \right) .
	\end{align}
	Then by non-negativity of $B_{\sigma}[Q]$, we have $Q\in C(\Reals^3\setminus \{0\})$ and 
	\begin{align}
		Q(x) \geq \theta \Phi_{\sigma} (x)>0,\quad  \text{for $0\neq x\in \Reals^3$},
	\end{align}
	so ~\eqref{Qpositive} holds. 
\end{proof}

%% file: Decay.tex
Next, we prove exponential decay of the solution $Q$ constructed above.

\begin{lemma}\label{lem:decay}
	Let $Q$ be the solution of \eqref{eq:fixedLemma} provided by Lemma~\ref{lem:fixed}. Let $\sigma$ be given by \eqref{eq:sigma} and
	$R= Q - \theta \Phi_\sigma$. Then $Q$ and $R$ satisfy the estimates:
	\begin{align} \label{est:QLemma}
	c \frac{e^{-\sqrt{\sigma}|x|}}{|x|}\leq Q(x) &\leq  \frac{Ce^{-\sqrt{\sigma}|x|}}{|x|}\\
	|R(x)|&\leq C e^{-\sqrt{\sigma}|x|}. \label{est:Rlemma}
	\end{align}
	Furthermore,
	\begin{align}\label{eq:continuous}
	R\in C(\Reals^3), \quad Q\in  W^{1,q}(\Reals^3)	\quad \text{for all \; $1 \le q < 3/2$}.
	\end{align}
\end{lemma}
\begin{proof}
	By construction, the solution $Q(x)= \tilde{Q}(|x|)$ is radial and nonnegative. Since $Q$ satisfies the equation \eqref{eq:fixedLemma}, we can represent $R$ as
	\begin{align} 
	R = (\sigma \Id - \Delta)^{-1} \left(B_\sigma[Q]\right).
	\end{align}
	Therefore $R\in H^1_\mathrm{rad}(\R^3)$. The radial Sobolev embedding then yields (cf.~\eqref{lions:pointwise})
	\begin{align}
	0\leq R(x)\leq \frac{C}{|x|}.
	\end{align}
	We write the equation in the form:
	\begin{align} \label{eq:iteration}
	R = (\sigma \Id - \Delta)^{-1} (B_\sigma [R+\theta \Phi_\sigma]) = \Phi_{\sigma} *(B_\sigma [R+\theta \Phi_\sigma]).
	\end{align}
	Now we use that $B_\sigma$ satisfies~\eqref{est:BQ}, so for $G\in L^1_{loc}$:
	\begin{align}\label{eq:Bpointwise}
	|B_\sigma [G](x)| \leq \frac{A|G(x)|^2}{1+|G(x)|}, \quad \text{for some $A > 0$}.
	\end{align}
	Using this and the exponential decay of the fundamental solution $\Phi_\sigma$, the equation \eqref{eq:iteration} implies:
	\begin{align}
	|R(x)|\leq \frac{C}{(1+|x|)}.
	\end{align}
	Since $\Phi_{\sigma}\in W^{1,q}$, $1\le q<\frac32$, the Sobolev embedding yields $R\in C(\Reals^3)$.
	This allows us to argue by a {m}aximum principle as follows.
	The inequality \eqref{eq:Bpointwise} shows that for any $\delta>0$ we can find a $C_\delta>0$ such that:
	\begin{align}
	B_\sigma[R+\theta \Phi]\leq \delta R(x) + C_\delta(\theta \Phi_\sigma(x) +  \chi_{|x|\leq C_\delta}).
	\end{align}	
	{Selecting} $\delta= \sigma/4$  we find:
	\begin{align}
	&&(\sigma \Id - \Delta) R &\leq   \frac{\sigma}{4} R(x) + C_\delta(\theta \Phi_\sigma(x) +  \chi_{|x|\leq C_\delta}) 
	\end{align}
	We absorb $\frac{\sigma}{4}R(x)$ in the left-hand side to obtain
	\begin{align}
		&& (\frac34 \sigma \Id - \Delta) R&\leq  C_\delta(\theta \Phi_\sigma(x) +  \chi_{|x|\leq C_\delta}) \\
	\Rightarrow&&  R&\leq C_\delta \Phi_{\frac34 \sigma}*\left(\theta \Phi_\sigma(x) +  \chi_{|x|\leq C_\delta} \right).
	\end{align}
	Performing the convolution integral, $R(x)$ can {therefore} be estimated by:
	\begin{align}
	R(x)\leq C e^{-\frac34  \sqrt{\sigma}|x|}.
	\end{align}
	Inserting this estimate into \eqref{eq:Bpointwise} gives 
	\begin{align} \label{eq:Rimproved}
	B_\sigma [R+\theta \Phi_\sigma](x) \leq C \left(e^{-\frac64 \sqrt{\sigma}|x| }+ \min\{1/|x|, \frac{e^{-2\sqrt{\sigma}|x|}}{|x|^2}\}\right) =: \zeta(x).
	\end{align}
	We reinsert this estimate into \eqref{eq:iteration} to bound $R$ as:
	\begin{align}
	R(x) \leq C(\Phi_{\sigma} * \zeta)(x) \leq C e^{- \sqrt{\sigma}|x|}.
	\end{align}	
	Recalling that $Q=R+\theta \Phi_\sigma${,} we obtain the desired estimate for $Q$. The lower bound for $Q$ follows from \eqref{Qpositive}.
	Since $Q$ satisfies the fixed point equation~\eqref{eq:fixedLemma} and $(\sigma \id - \Delta)^{-1}$ maps $W^{-1,q} \to W^{1,q}$ (this follows because it is comparable to the Bessel potential), we conclude that $Q \in W^{1,q}$. 
\end{proof}  
The estimates above are not quantitative in the size $\theta>0$ of the point charge. The following Lemma gives the desired estimate in $\theta$.
\begin{lemma}\label{lem:theta}
	Let $Q_\theta$ be the solution to equation~\eqref{eq:fixedLemma} with point charge $\theta > 0$ provided by Lemma~\ref{lem:fixed}. Then
	\begin{align} \label{eq:QComparison}
	 Q_\theta(x) \le \theta Q_1(x) \quad {\text{for all  $x \in \R^3 \setminus \{0\}$.}}
	\end{align}
\end{lemma}

\begin{proof}
	By definition, $Q_\theta$ is a weak solution to the equation
	\[
	-\Delta Q_\theta = g(Q_\theta) - 1 + \theta \delta_0 \quad \text{on $\R^3$.}
	\]
	Let us fix $\theta > 0$ and set $q \coloneqq \theta  Q_1 - 
	Q_\theta$. 	We make the observation that $q$ is a weak solution of the equation
	\begin{align} \label{qEquation}
	-\Delta q = \theta g(Q_1) + (1- \theta) - g(Q_\theta)  .
	\end{align}
	Using the convexity of $g$ and the fact that $g(0) = 1$,
	we deduce that $q$ satisfies the weak differential inequality
	\begin{align}
	-\Delta q & \ge g(\theta Q_1)  - g(Q_\theta).
	\end{align}
	Since $q$ satisfies \eqref{qEquation}, we have $q\in C(\Reals^3)$ and therefore it suffices to show that $q\geq 0$ outside the origin. We argue by contradiction: assume that there exists $x_0\in \Reals^3 \setminus \{0\}$ such that $q(x_0)= -c<0$. Since $q$ is radial, integrable and continuous, we may find $R>|x_0|$ sufficiently large such that $|q(x)|\leq c/2$ for $|x|=R$.	 Define 
	\begin{align}
		U:=\{x\in \Reals^3:  q(x)< 0 \text{ and } |x|<R\}.
	\end{align}
 Notice that
by the monotonicity of $g$ (recall that $g' < 0$) we get
\[
	-\Delta q \ge g(\theta Q_1) - g(Q_\theta) \ge 0 \qquad \text{weakly on $U$}.
\]
	Then $\partial U\subset \{x\in \Reals^3: |x|=R\text{ or } q(x)=0\}$ and therefore the minimum principle gives
	\begin{align}\label{partialU}
		q(x)\geq -\frac{c}{2}>-c \quad \text{for all $x\in U$}.
	\end{align}
This posses  a contradiction to the fact that $x_0\in U$  and the assumption $q(x_0)=-c$. This finishes the proof.
\end{proof}

\begin{proofof}[Proof of Theorem \ref{Theorem}]
	We claim that the function
	\begin{align}
		f(x,v)= F(Q(x)+\frac12|v|^2),
	\end{align}
	is a solution of \eqref{eq:main} in the sense of Definition~\ref{def:weak}. Continuity of $f$ follows from \eqref{eq:continuous} and {the fact that} $Q(x) \rightarrow \infty$ as $|x|\rightarrow 0$. The boundary condition \eqref{def:boundary} holds since $|Q(x)|\rightarrow 0$ for $|x|\rightarrow \infty$ and the definition of $F$. Similarly, the local integrability condition
	\begin{align}
		\int_K \int_{\Reals^3} |f(x,v)| \ud{x}\ud{v}= \int_K \int_{\Reals^3} F(Q(x)+\frac12 |v|^2)\ud{x}\ud{v}= \int_K  g(Q(x)) \ud{x}<\infty,
	\end{align}
	is satisfied since $g$ is bounded. Hence the spatial density $\rho$ 
	\begin{align} \label{eq:rhoRep}
		\rho(x)= g(Q(x))
	\end{align}
	is well defined. By construction of $Q$ then 2) holds.
	Finally, $f$ solves equation~\eqref{eq:main} since $f$ is constructed as a function of the Hamiltonian:
	\begin{align}
		\int_{\Reals^3} \int_{\Reals^3}  f(x,v)  v\cdot \nabla_x \vphi \ud{x}\ud{v} &= - \int_{\Reals^3} \int_{\Reals^3} v \cdot \nabla_x f(x,v)\vphi \ud{x}\ud{v} \\
		&= -\int_{\Reals^3} \int_{\Reals^3} F'(Q(x)+\frac12 |v|^2) v \cdot \nabla_x Q(x) \vphi \ud{x}\ud{v}\\
		&= -\int_{\Reals^3} \int_{\Reals^3} \nabla_v f(x,v) \cdot \nabla_x Q(x) \vphi \ud{x}\ud{v}\\
		&= \int_{\Reals^3} \int_{\Reals^3}  f(x,v) \nabla_x Q(x) \cdot \nabla_v \vphi \ud{x}\ud{v}.
	\end{align}
	
	The pointwise estimate for $Q$  (cf. \eqref{est:Q}) follows by combining Lemma~\ref{lem:decay} and Lemma~\ref{lem:theta}. Then the estimate \eqref{est:rho} for $\rho$ follows from \eqref{eq:rhoRep} and the fact that
	\begin{align}
		0\leq g(Q(x))\leq 1,\quad g'(0) =-\sigma<0.
	\end{align}
	We actually obtain the stronger estimate:
	\begin{align}
		0\leq 1-\varrho[f]\leq \min\{1,C \theta e^{-\frac{|x|}{\lambda_D}}\}.
	\end{align}
		The fact that the screening length 
		coincides with the Debye length $\lambda_D$ (cf. \eqref{eq:Debye}) for Maxwellian distributions $f_0=M_T$ (cf. \eqref{eq:Maxwellian}) follows by the following identity for $\sigma$ (cf. \eqref{eq:sigma}):
	\begin{align}
	\sigma = -g'(0) &= -4\pi \int_0^\infty  {\sqrt{2r}} F'(r) \ud{r}=T^{-2} .
	\end{align}
	This finishes the proof.
\end{proofof}

\subsection*{Acknowledgments} The authors would like to thank the anonymous referees for their careful reading and suggestions, which led to a substantial improvement of this work. A.A.-R. has received funding from the European Research Council (ERC) under the European Union's Horizon 2020 research and innovation programme, grant agreement No 757254 (SINGULARITY). R.W. acknowledges support of the Universit\'{e}  de Lyon through the IDEXLYON Scientific Breakthrough Project `Particles drifting and propelling in turbulent flows', and the hospitality of the UMPA, ENS Lyon.